\documentclass[a4paper,12pt]{amsart}

\usepackage{amsfonts,amssymb,amscd,amsmath,latexsym,amsbsy,enumerate,stmaryrd,a4wide,verbatim,color}
\usepackage{hyperref}

\theoremstyle{plain}
\newtheorem{theorem}{Theorem}[section]
\newtheorem{corollary}[theorem]{Corollary}

\newtheorem{lemma}[theorem]{Lemma}
\newtheorem{proposition}[theorem]{Proposition}
\newtheorem{Definition}[theorem]{Definition}
\theoremstyle{remark}
\newtheorem{remark}[theorem]{Remark}
\numberwithin{equation}{section}

\newcommand{\R}{\mathbb R}
\newcommand{\N}{\mathbb N}
\newcommand{\C}{\mathbb C}
\newcommand{\Z}{\mathbb Z}

\newcommand{\al}{\alpha}
\newcommand{\be}{\beta}

\newcommand{\Ga}{\Gamma}
\newcommand{\de}{\delta}

\newcommand{\eps}{\varepsilon}
\newcommand{\si}{\sigma}

\newcommand{\la}{\lambda}
\newcommand{\La}{\Lambda}

\newcommand{\om}{\omega}

\DeclareMathOperator{\sgn}{sgn}
\DeclareMathOperator{\arccosh}{arccosh}

\newcommand{\rFs}[5]{\,_{#1}F_{#2} \left( \genfrac{.}{.}{0pt}{}{#3}{#4}
	\ ;#5 \right)}

\newcommand{\su}{\mathfrak{su}}

\newcommand{\mhyphen}{\text{--}}

\begin{document}
\title{Orthogonal functions related to Lax pairs in Lie algebras}
\author{Wolter Groenevelt}
\address{Delft University of Technology\\
	Delft Institute of Applied Mathematics\\
	P.O.~Box 5031\\
	2600 GA Delft\\
	The Netherlands}
\email{w.g.m.groenevelt@tudelft.nl}
\author{Erik Koelink}
\address{ 
	IMAPP\\ Radboud Universiteit\\ P.O.~Box 9010\\ 6500 GL Nijmegen\\The Netherlands
}
\email{e.koelink@math.ru.nl}

\date{\today}
\dedicatory{Dedicated to the memory of Richard Askey}

\begin{abstract}
We study a Lax pair in a $2$-parameter Lie algebra in various representations. The 
overlap coefficients of the eigenfunctions of $L$ and the standard basis are given in terms of orthogonal
polynomials and orthogonal functions. Eigenfunctions for the operator $L$ for a 
Lax pair for $\mathfrak{sl}(d+1,\C)$ are studied 
in certain representations.
\end{abstract}

\maketitle

\section{Introduction}

The link of the Toda lattice to three-term recurrence relations via the Lax pair after
the Flaschka coordinate transform is well understood, see e.g. 
\cite{BabeBT}, \cite{Tesc}. 
We consider a Lax pair in a specific Lie algebra, such that in irreducible $*$-representations the Lax operator is a Jacobi operator. 
A Lax pair is a pair of time-dependent matrices or operators $L(t)$ and $M(t)$ satisfying the Lax equation
\[
\dot{L}(t) = [M(t),L(t)],
\]
where $[\, , \, ]$ is the commutator and the dot represents differentiation with respect to time. The Lax operator $L$ is isospectral, i.e.~the spectrum of $L$ is independent of time. A famous example is the Lax pair for the Toda chain in which $L$ is a self-adjoint Jacobi operator,
\[
L(t) e_n = a_n(t) e_{n+1} + b_n(t) e_{n} + a_{n-1}(t) e_{n-1},
\]
where $\{e_n\}$ is an orthonormal basis for the Hilbert space, and $M$ is the skew-adjoint operator given by
\[
M(t) e_n = a_n(t) e_{n+1} - a_{n-1}(t) e_{n-1}.
\] 
In this case the Lax equation describes the equations of motion (after a change of variables) of a chain of interacting particles with nearest neighbour interactions. The eigenvalues of $L$, $L$ being isospectral, constitute integrals of motion.

In this paper we define a Lax pair in a 2-parameter Lie algebra. In the special case of $\mathfrak{sl}(2,\C)$ we recover the Lax pair for the $\mathfrak{sl}(2,\C)$ Kostant Toda lattice, see \cite[\S 4.6]{BabeBT} and 
references given there. We give a slight generalization by allowing for a more general $M(t)$. 
We discuss the corresponding solutions to the corresponding differential 
equations in various representations of the Lie algebra.
In particular, one obtains the classical relation to the Hermite, Krawtchouk, Charlier, Meixner, Laguerre and Meixner-Pollaczek polynomials from the Askey scheme of hypergeometric functions  \cite{KLS} for which the 
Toda modification, see \cite[\S 2.8]{Isma}, remains in the same class of orthogonal polynomials. 
This corresponds to the results 
established by Zhedanov \cite{Zhed}, who investigated the situation where $L$, $M$ and $\dot{L}$ act as
three-term recurrence operators and close up to a Lie algebra of dimension $3$ or $4$. In the current paper
Zhedanov's result is explained, starting from the other end. In Zhedanov's approach the condition on 
forming a low-dimensional Lie algebra forces a factorization of the functions as a function of time $t$ and place $n$, which is immediate from representing the Lax pair from the Lie algebra element. 
The solutions of the Toda lattice arising in this way, i.e.  which are factorizable as functions of $n$ and $t$, have 
also been obtained by Kametaka \cite{Kame} stressing 
the hypergeometric nature of the solutions. The 
link to Lie algebras and Lie groups in Kametaka \cite{Kame} is 
implicit, see especially \cite[Part I]{Kame}. 
The results and methods of the short paper by Kametaka \cite{Kame} have been explained and extended later
by Okamoto \cite{Okam}. In particular, Okamoto \cite{Okam} gives the relation to the $\tau$-function formulation and 
the B\"acklund transformations. 

Moreover, we extend to non-polynomial
settings by considering  representations
of the corresponding Lie algebras in $\ell^2(\Z)$ corresponding to the principal unitary series 
of $\mathfrak{su}(1,1)$ and the representations of $\mathfrak{e}(2)$, the Lie algebra of the group 
of motions of the plane. In this way we find solutions to the Toda lattice equations labelled by $\Z$.
There is a (non-canonical) way to associate to recurrences on $\ell^2(\Z)$ three-term recurrences
for $2\times 2$-matrix valued polynomials, see e.g. \cite{Bere}, \cite{Koel}. However, this does
not lead to explicit $2\times 2$-matrix valued solutions of the non-abelian Toda lattice as introduced and 
studied in \cite{BrusMRL}, \cite{Gekh} in relation to matrix valued orthogonal polynomials, see 
also \cite{IsmaKR} for an explicit example and the relation to the modification of the matrix weight.
The general Lax pair for the Toda lattice in finite dimensions, as studied by Moser
\cite{Mose}, can also be considered and slightly extended in the same way as an element of
the Lie algebra $\mathfrak{s\l}(d+1,\C)$. This involves $t$-dependent 
finite discrete orthogonal polynomials, and these polynomials occur in describing the action of 
$L(t)$ in highest weight representations. We restrict to representations for the symmetric powers
of the fundamental representations, then the eigenfunctions can be described in 
terms of multivariable Krawtchouk polynomials following Iliev \cite{Ilie} establishing them as overlap coefficients between a natural basis for two different Cartan subalgebras. Similar group theoretic interpretations of these multivariable Krawtchouk polynomials have been established by 
Cramp\'e et al. \cite{CramvdVV} and Genest et al. \cite{GeneVZ}. 
We discuss briefly the $t$-dependence of the corresponding eigenvectors of $L(t)$. 

In brief, in \S \ref{sect:g(a,b)} we recall the $2$-parameter Lie algebra as in \cite{Mi} and the 
Lax pair. In \S \ref{sec:su2} we discuss $\mathfrak{su}(2)$ and its finite-dimensional representations,
and in \S \ref{sec:su11} we discuss the case of $\mathfrak{su}(1,1)$, where we discuss both
discrete series representations and principal unitary series representations. The last leads to new 
solutions of the Toda equations and the generalization in terms of orthogonal functions. 
The corresponding orthogonal functions are the overlap coefficients between the standard basis in 
the representations and the $t$-dependent eigenfunctions of the operator $L$. 
In \S \ref{sec:oscillatoralgebra} we look at the oscillator algebra as specialization, and in 
\S \ref{sec:casee2} we consider the Lie algebra for the group of plane motions leading to a solution in 
connection to Bessel functions. 
In \S \ref{sec:modification} we indicate how the measures for the orthogonal functions involved have to be modified
in order to give solutions of the coupled differential equations. For the Toda case related to orthogonal
polynomials, this coincides with the Toda modification \cite[\S 2.8]{Isma}. 
Finally, in \S \ref{sec:sld+1} we consider the case of finite dimensional representations of such a 
Lax pair for a higher rank Lie algebra in specific finite-dimensional representations for which 
all weight spaces are $1$-dimensional. 

A question following up on \S \ref{sec:modification} is whether the modification for the weight is 
of general interest, cf. \cite[\S 2.8]{Isma}. 
A natural question following up on \S \ref{sec:sld+1} is what happens in other finite-dimensional 
representations, and what 
happens in infinite dimensional representations corresponding to non-compact real forms of $\mathfrak{sl}(d+1,\C)$ as is done in \S \ref{sec:su11} for the case $d=1$. We could also ask if it is possible to associate 
Racah polynomials, as the most general finite discrete orthogonal polynomials in the Askey scheme, to 
the construction of \S \ref{sec:sld+1}. Moreover, the relation to the interpretation as 
in \cite{KVdJ98} suggests that it might be possible to extend to quantum algebra setting, but this is quite
open.

\subsection*{Dedication} This paper is dedicated to Richard A. Askey (1933--2019) who has done an 
incredible amount of fascinating work in the area of special functions, and who always had an open mind, 
in particular concerning  
relations with other areas. We hope this spirit is reflected in this paper. Moreover, through his efforts 
for mathematics education, Askey's legacy will be long-lived 

\subsection*{Acknowledgement} We thank Luc Vinet for 
pointing out references. We also thank both referees
for their comments, and especially for pointing out the 
papers by Kametaka \cite{Kame} and Okamoto \cite{Okam}. 

\section{The Lie algebra $\mathfrak g(a,b)$} \label{sect:g(a,b)}
Let $a,b \in \C$.  The Lie algebra $\mathfrak g(a,b)$ is the 4-dimensional complex Lie algebra with basis $H,E,F,N$ satisfying
\begin{equation} \label{eq:commuation relations}
\begin{gathered}
{} [E,F]=aH+bN, \quad [H,E]=2E, \quad [H,F]=-2F, \\ [H,N]=[E,N]=[F,N]=0.
\end{gathered}
\end{equation}
For $a,b \in \R$ there are two inequivalent $*$-structures on $\mathfrak g(a,b)$ defined by
\[
E^*=\epsilon F, \quad H^*=H, \quad N^*=N,
\]
where $\epsilon \in \{+,-\}$.

We define the following Lax pair in $\mathfrak g(a,b)$. 
\begin{Definition} 
	Let $r,s \in C^1 [0,\infty)$ and $u \in C[0,\infty)$ be real-valued functions and let $c \in \R$. The Lax pair $L,M \in \mathfrak g(a,b)$ is given by
	\begin{equation} \label{eq:Lax pair}
		\begin{split}
			L(t) &= cH+ s(t)(aH+bN)+r(t) \big(E+E^*\big), \\ M(t)&= u(t)\big(E-E^*\big). 
		\end{split}
	\end{equation}
\end{Definition}
Note that $L^*=L$ and $M^*=-M$. Being a Lax pair means that $\dot L = [M,L]$, which leads to the following differential equations.
\begin{proposition} \label{prop:relations r s u}
	The functions $r,s$ and $u$ satisfy
	\[
	\begin{split}
	\dot s(t) & = 2\epsilon r(t)u(t)\\
	\dot r(t) & = -2 (as(t)+c)u(t).
	\end{split}
	\]
\end{proposition}
\begin{proof}
From the commutation relations \eqref{eq:commuation relations} it follows that
\[
[M,L]= 2\epsilon r(t)u(t)(aH+bN) - 2 (as(t)+c)u(t)(E+E^*).
\]
Since $[M,L]=\dot L = \dot s(t)(aH+bN) + \dot r(t) (E+E^*)$, the results follows.
\end{proof}

\begin{corollary} \label{cor:constant function}
The function $I(r,s)=\epsilon r^2+ (as+2c)s$ is an invariant.
\end{corollary}
\begin{proof}
	Differentiating gives
	\[
	\begin{split}
	\frac{d}{dt} (\epsilon r(t)^2+as(t)^2+2cs(t)) &= 2\epsilon r(t) \dot r(t)+ 2 (as(t)+c) \dot s(t),
	 \end{split}
	 \]
	which equals zero by Proposition \ref{prop:relations r s u}.
\end{proof}

In the following sections we consider the Lax operator $L$ in an irreducible $*$-representation of $\mathfrak g(a,b)$, and we determine explicit eigenfunctions and its spectrum. We restrict to the following special cases of the Lie algebra $\mathfrak g(a,b)$:
\begin{itemize}
	\item $\mathfrak g(1,0) \cong \mathfrak{sl}(2,\C)\oplus \C$
	\item $\mathfrak g(0,1) \cong \mathfrak b(1)$ is the generalized oscillator algebra
	\item $\mathfrak g(0,0) \cong \mathfrak e(2) \oplus \C$, with $\mathfrak e(2)$ the Lie algebra of the group of plane motions
\end{itemize}
These are the only essential cases as $\mathfrak g(a,b)$ is isomorphic as a Lie algebra to one of these cases, see \cite[Section 2-5]{Mi}.

\section{The Lie algebra $\su(2)$}\label{sec:su2}

In this section we consider the Lie algebra $\mathfrak g(a,b)$ from Section \ref{sect:g(a,b)} with  $(a,b)=(1,0)$ and $\epsilon = +$, i.e.~the Lie algebra $\su(2) \oplus \C$. The basis element $N$ plays no role in this case, therefore we omit it. So we consider the Lie algebra with basis $H,E,F$ satisfying commutation relations
\[
[H,E]=2E, \qquad [H,F]=-2F, \qquad [E,F]=H,
\]
and the $*$-structure is defined by $H^*=H, E^*=F$.

The Lax pair \eqref{eq:Lax pair} is given by
\[
L(t) = s(t)H + r(t)(E+F), \qquad M(t)=u(t)(E-F),
\]
where (without loss of generality) we set $c=0$. The differential equations for $r$ and $s$ from Proposition \ref{prop:relations r s u} read in this case
\begin{equation} \label{eq:relations r s u su(2)}
	\begin{split}
		\dot s(t) &= 2u(t)r(t) \\
		\dot r(t) & = -2u(t)s(t)
	\end{split}
\end{equation}
and the invariant in Corollary \ref*{cor:constant function} is given by $I(r,s)=r^2+s^2$.

\begin{lemma} \label{lem:sign r s for su(2)}
	Assume $\sgn(u(t))=\sgn(r(t)$ for all $t>0$, $s(0)>0$ and $r(0)>0$. Then $\sgn(s(t))>0$ and $\sgn(r(t))>0$ for all $t>0$. 
\end{lemma}
\begin{proof}
	From $\dot s = 2 ur$ it follows that $s$ is increasing. Since $(r(t),s(t))$ in phase space is a point on the invariant $I(r,s)=I(r(0),s(0))$, which describes a circle around the origin, it follows that $r(t)$ and $s(t)$ remain positive.
\end{proof}
Throughout this section we assume that the conditions of Lemma \ref{lem:sign r s for su(2)} are satisfied, so that $r(t)$ and $s(t)$ are positive. Note that if we change the condition on $r(0)$ to $r(0)<0$, then $r(t)<0$ for all $t>0$.

For $j \in \frac12\N$ let $\ell^2_j$ be the $2j+1$ dimensional complex Hilbert space with standard orthonormal basis $\{e_n \mid n=0,\ldots,2j\}$. An irreducible $*$-representation $\pi_j$ of $\su(2)$ on $\ell^2_j$ is given by
\[
\begin{split}
\pi_j(H)e_n & = 2(n-j)\, e_n \\
\pi_j(E) e_n & = \sqrt{(n+1)(2j-n))}\, e_{n+1} \\
\pi_j(F) e_n & = \sqrt{n(2j-n+1)}\, e_{n-1},
\end{split}
\]
where we use the notation $e_{-1}=e_{2j+1}=0$. In this representation the Lax operator $\pi_j(L)$ is the Jacobi operator
\begin{equation} \label{eq:L Krawtchouck}
\pi_j(L(t)) e_n = r(t) \sqrt{(n+1)(2j-n)} \, e_{n+1} + 2s(t)(n-j) \, e_n + r(t) \sqrt{n(2j-n+1)}\, e_{n-1}.
\end{equation}

We can diagonalize the Lax operator $\pi_j(L)$ using orthonormal Krawtchouk polynomials \cite[Section 9.11]{KLS}, which are defined by
\[
K_n(x) = K_n(x;p,N) = \left(\frac{p}{1-p}\right)^\frac{n}{2} \sqrt{\binom{N}{n}} \rFs{2}{1}{-n,-x}{-N}{\frac{1}{p}},
\]
where $N\in \N$, $0<p<1$ and $n,x \in  \{0,1,\ldots,N\}$. The three-term recurrence relation is
\[
\begin{split}
&\frac{\frac12 N- x}{\sqrt{p(1-p)}}K_n(x) = \\ &\quad \sqrt{(n+1)(N-n)}\, K_{n+1}(x) + \frac{p-\frac12}{\sqrt{p(1-p)}}(2n-N)K_n(x) + \sqrt{n(N-n+1)}\, K_{n-1}(x),
\end{split}
\]
with the convention $K_{-1}(x) = K_{N+1}(x)=0$. The 
orthogonality relations read 
\[
\sum_{x=0}^N \binom{N}{x} p^x (1-p)^{N-x} K_{n}(x) K_{n'}(x) = \de_{n,n'}.
\]

\begin{theorem} \label{thm:Krawtchouk}
Define for $x \in \{0,\ldots,2j\}$
\[
W_t(x) = \binom{2j}{x}p(t)^x(1-p(t))^{2j-x}
\]
where $p(t) = \frac12 + \frac{s(t)}{2C}$ and $C = \sqrt{s^2 + r^2}$.  For $t>0$ let $U_t: \ell^2_j \to \ell^2(\{0,\ldots,2j\}, W_t)$  be defined by
\[
[U_te_n](x) = K_n(x;p(t),2j),
\]
then $U_t$ is unitary and $U_t \circ \pi_j(L(t)) \circ U_t^* = M(2C(j-x))$.
\end{theorem}
Here $M$ denotes the multiplication operator given by $[M(f)g](x) = f(x)g(x)$. 
\begin{proof}
	From \eqref{eq:L Krawtchouck} and the recurrence relation of the Krawtchouk polynomials we obtain
	\[
	[U_t\,r^{-1}\pi_j(L) U_t^* K_\cdot(x)](n) = \frac{j-x}{\sqrt{p(1-p)}} K_n(x),
	\]
	where 
	\[
	\frac{s}{r} = \frac{p-\frac12}{\sqrt{p(1-p)}}.
	\]
	The last identity implies
	\[
	p= \frac12+ \frac12\sqrt{\frac{s^2}{s^2+r^2}},
	\]
	so that
	\[
	p(1-p)= \frac{r^2}{4(s^2+r^2)}.
	\]
	Then we find that the eigenvalue is 
	\[
	\frac{j-x}{\sqrt{p(1-p)}}= \frac{\sqrt{s^2+r^2}}{r} 2(j-x).
	\]
	Since $s^2+r^2$ is constant, the result follows.
\end{proof}

\section{The Lie algebra $\su(1,1)$}\label{sec:su11}

In this section we consider representations of $\mathfrak g(a,b)$ with $(a,b)=(1,0)$ and $\epsilon=-$, i.e.~the Lie algebra $\su(1,1) \oplus \C$. We omit the basis element $N$ again. The commutation relations are the same as in the previous section. The $*$-structure in this case is defined by  $H^*=H$ and $E^*=-F$.

The Lax pair \eqref{eq:Lax pair} is given by
\[
L(t) = s(t)H + r(t)(E-F), \qquad M(t)=u(t)(E+F),
\]
where we set $c=0$ again. The functions $r$ and $s$ satisfy
	\[
	\begin{split}
	\dot s(t) &= -2u(t)r(t) \\
	\dot r(t) & = -2u(t)s(t)
	\end{split}
	\]
and the invariant is given by $I(r,s)=s^2-r^2$.

\begin{lemma} \label{lem:sign r s for su(1,1)}
	Assume $\sgn(u(t))=-\sgn(r(t)$ for all $t>0$, $s(0)>0$ and $r(0)>0$. Then $\sgn(s(t))>0$ and $\sgn(r(t))>0$ for all $t>0$. 
\end{lemma}
\begin{proof}
	The proof is similar to the proof of Lemma \ref{lem:sign r s for su(2)}, where in this case  $I(r,s)=I(r(0),s(0))$ describes a hyperbola or a straight line.
\end{proof}
Throughout this section we assume that the assumptions of Lemma \ref{lem:sign r s for su(1,1)} are satisfied.\\

We consider two families of irreducible $*$-representations of $\su(1,1)$. 
The first family is the positive discrete series representations $\pi_k$, $k>0$, on $\ell^2(\N)$. The actions of the basis elements on the standard orthonormal basis $\{e_n \mid n \in \N\}$ are given by
\[
\begin{split}
	\pi_k(H)e_n &= 2(k+n)\, e_n \\
	\pi_k(E)e_n & = \sqrt{(n+1)(2k+n)}\, e_{n+1} \\
	\pi_k(F)e_n & = -\sqrt{n(2k+n-1)}\, e_{n-1}.
\end{split}
\]
We use the convention $e_{-1}=0$.

The second family of representations we consider is the principal unitary series representation $\pi_{\la,\eps}$, $\la \in -\frac12+i\R_+$, $\eps \in [0,1)$ with $(\la,\eps) \neq (-\frac12,\frac12)$, on $\ell^2(\Z)$. The actions of the basis elements on the standard orthonormal basis $\{ e_n \mid n \in \Z \}$ are given by
\[
\begin{split}
	\pi_{\la,\eps}(H)e_n &= 2(\eps+n)\, e_n \\
	\pi_{\la,\eps}(E)e_n & = \sqrt{(n+\eps-\la)(n+\eps+\la+1)}\, e_{n+1} \\
	\pi_{\la,\eps}(F)e_n & = -\sqrt{n+\eps-\la-1)(n+\eps+\la)}\, e_{n-1}.
\end{split}
\]

Note that both representations $\pi_k^+$ and $\pi_{\la,\eps}$ as given above define unbounded representations. The operators $\pi(X)$, $X \in \su(1,1)$, are densely defined operators on their representation space, where as a dense domain we take the set of finite linear combinations of the standard orthonormal basis $\{e_n\}$.

\begin{remark}
	The Lie algebra $\su(1,1)$ has two more families of irreducible $*$-representations: the negative discrete series and the complementary series. The negative discrete series representation $\pi_k^-$, $k>0$, can be obtained from the positive discrete series representation $\pi_k$ by setting 
	\[
	\pi_k^-(X) = \pi_k(\vartheta(X)), \qquad X \in \su(1,1)
	\]
	where $\vartheta$ is the Lie algebra isomorphism defined by $\vartheta(H)=-H$, $\vartheta(E)=F$, $\vartheta(F)=E$. 
	
	The complementary series are defined in the same way as the principal unitary series, but the labels $\la,\eps$ satisfy $\eps \in [0,\frac12)$, $\la \in (-\frac12,-\eps)$ or $\eps \in (\frac12,1)$, $\la \in (-\frac12, \eps-1)$.
	
	The results obtained in this section about the Lax operator in the positive discrete series and principal unitary series representations can easily be extended to these two families of representations.
\end{remark}

\subsection{The Lax operator in the positive discrete series}
The Lax operator $L$ acts in the positive discrete series representation as a Jacobi operator on $\ell^2(\N)$ by
\[
 \pi_k(L(t)) e_n = r(t)\sqrt{(n+1)(n+2k)}\, e_{n+1} + s(t)(2k+2n) e_n + r(t)\sqrt{n(n+2k-1)}\, e_{n-1}.
\]
$\pi_k(L)$ can be diagonalized using explicit families of orthogonal polynomials. We need to distinguish between three cases corresponding to the invariant $s^2-r^2$ being positive, zero or negative. This 
corresponds to hyperbolic, parabolic and elliptic elements, and the eigenvalues and eigenfunctions 
have different behaviour per class, cf. \cite{KVdJ98}.

\subsubsection{Case 1: $s^2-r^2>0$}
In this case eigenfunctions of $\pi_k(L)$ can be given in terms of Meixner polynomials.
The orthonormal Meixner polynomials \cite[Section 9.10]{KLS} are defined by
\[
M_n(x)= M_n(x;\be,c) = (-1)^n \sqrt{ \frac{(\be)_n }{n!}c^n} \rFs{2}{1}{-n,-x}{\be}{1-\frac{1}{c}},
\]
where $\be>0$ and $0<c<1$. They satisfy the three-term recurrence relation
\[
\begin{split}
&\frac{(1-c)(x+\frac12\be)}{\sqrt c} M_n(x) = \\ &\sqrt{(n+1)(n+\be)} M_{n+1}(x) + \frac{ (c+1)(n+\frac12\be) }{\sqrt c} M_n(x) + \sqrt{n(n-1+\be)} M_{n-1}(x).
\end{split}
\]
Their orthogonality relations are given by
\[
\sum_{x\in \N} \frac{ (\be)_x }{x!}c^x  (1-c)^{2\be}   M_n(x) M_{n'}(x)  = \de_{n,n'}.
\]

\begin{theorem} \label{thm:Meixner}
	Let
	\[
	W_t(x) = \frac{ (2k)_x }{x!}c(t)^x  (1-c(t))^{4k}, \qquad x \in \N,\, t>0,
	\]
	where $c(t) \in (0,1)$ is determined by $\frac{s}{r}=\frac{1+c}{2\sqrt c}$, or equivalently $c(t) = e^{-2 \arccosh(\frac{s(t)}{r(t)})}$. Define for $t>0$ the operator $U_t:\ell^2(\N) \to \ell^2(\N,W_t)$ by
	\[
	[U_te_n](x) =  M_n(x;2k,c(t)),
	\]
	then $U_t$ is unitary and $U_t\circ \pi_k(L(t)) \circ U_t^* = M(2C(x+k))$ where $C = \sqrt{s^2-r^2}$.
\end{theorem}
\begin{proof}
The proof runs along the same lines as the proof of Theorem \ref{thm:Krawtchouk}.
The condition $s^2-r^2>0$ implies $s/r>1$, so there exists a $c = c(t) \in (0,1)$ such that 
\[
\frac{s}{r} =  \frac{ 1+c }{2\sqrt c}.
\] 
It follows from the three-term recurrence relation for Meixner polynomials that
$r^{-1}L$ has eigenvalues $\frac{(1-c)(x+k)}{\sqrt c}$, $x \in \N$. Write $c=e^{-2a}$ with $a>0$, then $\frac{ 1+c }{2\sqrt c}= \cosh(a)$, so that
\[
\frac{1-c}{2\sqrt{c}} = \sinh(a)= \sqrt{\cosh^2(a) -1 } = \sqrt{\frac{s^2}{r^2}-1} = \frac{C}{r},
\]
where $C =  \sqrt{s^2 -r^2}$. 
\end{proof}

\subsection{Case 2: $s^2-r^2=0$}
In this case we need the orthonormal Laguerre polynomials \cite[Section 9.12]{KLS}, which are defined by
\[
L_n(x)= L_n(x;\alpha) = (-1)^n \sqrt{ \frac{(\al+1)_n}{n!} } \rFs{1}{1}{-n}{\al+1}{x}. 
\]
They satisfy the three-term recurrence relation
\[
x L_n(x) = \sqrt{(n+\al+1)(n+1)}\, L_{n+1}(x) + (2n+\al+1)L_n(x) + \sqrt{n(n+\al)}\, L_{n-1}(x),
\]
and the orthogonality relations are
\[
\int_0^\infty L_n(x) L_{n'}(x) \, \frac{x^\al e^{-x}}{\Ga(\al+1)}\, dx = \de_{n,n'}.
\]
The set $\{L_n \mid n \in \N\}$ is an orthonormal basis for the corresponding weighted $L^2$-space. 

Using the three-term recurrence relation for the Laguerre polynomials we obtain the following result.

\begin{theorem} \label{thm:Laguerre}
	Let 
	\[
	W_t(x) = \frac{x^{2k-1} r(t)^{-2k} e^{-\frac{x}{r(t)}} }{\Ga(2k)},\qquad x \in [0,\infty),
	\]
	and let $U_t:\ell^2(\N) \to L^2([0,\infty),W_t(x)dx)$ be defined by
	\[
	[U_te_n](x) = L_n\left(\frac{x}{r(t)};2k-1\right),
	\]
	then $U_t$ is unitary and $U_t \circ \pi_k(L(t)) \circ U_t^{*} = M(x)$.
\end{theorem}

\subsection{Case 3: $s^2-r^2<0$}
In this case we need the orthonormal Meixner-Pollaczek polynomials \cite[Section 9.7]{KLS} given by
\[
P_n(x) = P_n(x;\la,\phi) = e^{in\phi} \sqrt{ \frac{(2\la)_n}{n!}} \rFs{2}{1}{-n, \la+ix}{2\la}{1-e^{-2i\phi}},
\]
where $\la>0$ and $0<\phi<\pi$. The three-term recurrence relation for these polynomials is
\[
2x\sin\phi\, P_n(x) = \sqrt{(n+1)(n+2k)}\,P_{n+1}(x) - 2(n+\la)\cos\phi \, P_n(x) + \sqrt{n(n+2k-1)}\, P_{n-1}(x),
\]
and the orthogonality relations read
\[
\begin{gathered}
	\int_{-\infty}^\infty P_n(x) P_{n'}(x)\, w(x;\la,\phi)\, dx = \de_{n,n'},\\
	w(x;\la,\phi) = \frac{(2\sin\phi)^{2\la}}{2\pi\,\Ga(2\la)} e^{(2\phi-\pi)x} |\Ga(\la+ix)|^2. 
\end{gathered}
\]
The set $\{P_n \mid n \in \N\}$ is an orthonormal basis for the weighted $L^2$-space.

\begin{theorem} \label{thm:Meixner-Pollaczek}
	For $\phi(t) = \arccos(\frac{s(t)}{r(t)})$ let
	\[
	W_t(x) = w(x;k,\phi(t)), \qquad x \in \R,
	\]
	and let $U_t : \ell^2(\N) \to L^2(\R,W_t(x)dx)$ be defined by
	\[
	[U_t e_n](x) = P_n(x;k,\phi(t))
	\]
	then $U_t$ is unitary and $U_t \circ \pi_k(L(t)) \circ U_t^{*} = M(-2Cx)$, where $C = \sqrt{r^2-s^2}$.
\end{theorem}
\begin{proof}
	The proof is similar as before. Using the three-term recurrence relation for the Meixner-Pollaczek polynomials it follows that the generalized eigenvalue of $r^{-1}\pi_k(L)$ is $- 2x \sin(\phi)$, where $\phi \in (0,\pi)$ is determined by $-\frac{s}{r} = \cos\phi$. Then
	\[
	\sin\phi = \sqrt{1-\frac{s^2}{r^2}} = \frac{C}{r},
	\]
	from which the result follows.
\end{proof}

\subsection{The Lax operator in the principal unitary series}
The action of the Lax operator $L$ in the principal unitary series as a Jacobi operator on $\ell^2(\Z)$ is given by
\[
\begin{split}
\pi_{\la,\eps} (L(t)) e_n &= r(t)\sqrt{(n+\eps+\la+1)(n+\eps-\la)}\, e_{n+1} + s(t)(2\eps+2n) e_n \\
& \quad + r(t)\sqrt{ (n+\eps+\la)(n+\eps-\la-1)}\, e_{n-1}.
\end{split}
\]
Again we distinguish between the cases where the invariant $s^2-r^2$ is either positive, negative or zero.

\subsubsection{Case 1: $s^2-r^2>0$}
The Meixner functions \cite{GroeK2002} are defined by
\[
\begin{split}
	m_n(x) = m_n(x;\la,\eps,c) &= \left(\frac{\sqrt{c}}{c-1}\right)^n \frac{ \sqrt{ \Ga(n+\eps+\la+1) \Ga(n+\eps-\la) } }{(1-c)^\eps \Ga(n+1-x)} \\
	& \quad \times \rFs{2}{1}{n+\eps+\la+1,n+\eps-\la}{n+1-x}{\frac{c}{c-1}},
\end{split}
\]
for $x,n \in \Z$ and $c \in (0,1)$. The parameters $\la$ and $\eps$ are the labels from the principal unitary series. The Meixner functions satisfy the three-term recurrence relation
\[
\begin{split}
	\frac{(1-c)(x+\eps)}{\sqrt{c}}m_n(x) & = \sqrt{(n+\eps+\la+1)(n+\eps-\la)}\, m_{n+1}(x) \\
	& \quad + \frac{(c+1)(x+\eps)}{\sqrt{c}} m_n(x) + \sqrt{(n+\eps+\la)(n+\eps-\la-1)} \, m_{n-1}(x),
\end{split}
\]
and the orthogonality relations read
\[
\sum_{x \in \Z} \frac{ c^{-x} }{ \Ga(x+ \eps + \la +1) \Ga(x+\eps-\la) } m_n(x) m_{n'}(x) = \de_{n,n'}.
\]
The set $\{ m_n \mid n \in \Z\}$ is an orthonormal basis for the weighted $L^2$-space. 
\begin{theorem} \label{thm:Meixner functions}
	For $t>0$ let
	\[
	W_t(x) =  \frac{ c(t)^{-x} }{ \Ga(x+ \eps + \la +1) \Ga(x+\eps-\la)},
	\]
	where  $c(t) \in (0,1)$ is determined by $\frac{s(t)}{r(t)}=\frac{1+c(t)}{2\sqrt{c(t)}}$, or equivalently $c(t) = e^{-2 \arccosh(\frac{s(t)}{r(t)})}$. Define $U_t:\ell^2(\Z) \to \ell^2(\Z,W_t)$ by
	\[
	[U_t e_n](x) = m_n(x;\la,\eps,c),
	\]
	then $U_t$ is unitary and $U_t \circ \pi_{\la,\eps}(L(t)) \circ U_t^* = M(2C(x+\eps))$, where $C = \sqrt{s^2 - r^2}$.
\end{theorem}

\subsubsection{Case 2: $s^2-r^2=0$} In this case we need Laguerre functions \cite{Groe2003} defined by
\[
\psi_n(x) = \psi_n(x;\la,\eps) = 
\begin{cases}
\begin{split}
	(-1)^n & \sqrt{\Ga(n+\eps+\la+1) \Ga(n+\eps-\la)} \\
&\quad \times  U(n+\eps+\la+1;2\la+2;x)
\end{split}
 & x>0\\  \\
\begin{split}
&\sqrt{ \Ga(-n-\eps-\la) \Ga(1-n-\eps+\la)} \\ 
& \quad \times  U(-n-\eps-\la;-2\la;-x)
\end{split}
 & x<0,
\end{cases}
\]
where $x\in\R$, $n \in \Z$, and $U(a;b;z)$ is Tricomi's confluent hypergeometric function, see e.g.~
\cite[(1.3.1)]{Sl}, for which we use its principal branch with branch cut along the negative real axis. 
The Laguerre functions $\{\psi_n \mid n \in \Z\}$ form an orthonormal basis for $L^2(\R,w(x)dx)$ where
\[
w(x)=w(x;\rho,\eps) = \frac1{\pi^2} \sin\left( \pi(\eps+\la+1) \right) \sin \left(\pi( \eps-\la)\right)  e^{-|x|}.
\]
The three-term recurrence relation reads
\[
\begin{split}
-x \psi_n(x) &= \sqrt{(n+\eps+\la+1)(n+\eps-\la)} \,\psi_{n+1}(x) \\
& \quad + 2(n+\eps)\, \psi_n(x) + \sqrt{(n+\eps+\la)(n+\eps-\la-1)}\, \psi_{n-1}(x).
\end{split}
\]
\begin{theorem} \label{thmL Laguerre functions}
Let
\[
W_t(x) = \frac{1}{r(t)} w\left( \frac{x}{r(t)};\la,\eps\right), \qquad x \in \R,
\]
and let $U_t : \ell^2(\Z) \to L^2(\R,W_t(x)dx)$ be defined by
\[
[U_t e_n](x) =  \psi_n\left( \frac{x}{r(t)};\la,\eps\right),
\]
then $U_t$ is unitary and $U_t \circ \pi_{\la,\eps}(L(t)) \circ U_t^{*} = M(-x)$. 
\end{theorem}

\subsubsection{Case 3: $s^2-r^2<0$}
The Meixner-Pollaczek functions \cite[\S4.4]{Koel2004} are defined by
\[
\begin{split}
u_n(x) = u_n(x;\la,\eps,\phi) & = (2i\sin\phi)^{-n} \frac{ \sqrt{\Ga(n+1+\eps+\la)\Ga(n+\eps-\la) }}{\Ga(n+1+\eps-ix)} \\
& \quad \times \rFs{2}{1}{n+1+\eps+\la,n+\eps-\la}{n+1+\eps-ix}{\frac{1}{1-e^{-2i\phi}}}.
\end{split}
\]
Define
\[
W(x;\la,\eps,\phi)= w_0(x)\begin{pmatrix}
1 & - w_1(x) \\ -\overline{w_1}(x) & 1
\end{pmatrix}, \qquad x \in \R,
\]
where $\overline{f}(x) = \overline{f(x)}$ and
\[
\begin{split}
w_1(x;\la,\eps) &= \frac{ \Ga(\la+1+ix) \Ga(-\la+ix) }{\Ga(ix-\eps)\Ga(1+\eps-ix)}\\
w_0(x;\eps,\phi) &= (2\sin\phi)^{-2\eps} e^{(2\phi-\pi)x}.
\end{split}
\]
Let $L^2(\R,W(x)dx)$ be the Hilbert space consisting of functions $\R \to \C^2$ with inner product
\[
\langle f,g \rangle = \int_{-\infty}^\infty 
g^t(x) W(x) f(x)\, dx
\]
where $f^t(x)$ denotes the conjugate transpose of $f(x) \in \C^2$. The set  $\{(\begin{smallmatrix}u_n \\ \overline{u_n} \end{smallmatrix}) \mid n \in \Z\}$ is an orthonormal basis for $L^2(\R,W(x)dx)$. The three-term recurrence relation for the Meixner-Pollaczek functions is 
\[
\begin{split}
2x \sin\phi\, u_n(x) & = \sqrt{(n+\eps+\la+1)(n+\eps-\la)}\, u_{n+1}(x) \\
&\quad + 2(n+\eps)\cos\phi\, u_n(x) + \sqrt{(n+\eps+\la)(n+\eps-\la-1)}\, u_{n-1}(x).
\end{split}
\]
The function $\overline{u_n}$ satisfies the same recurrence relation.
\begin{theorem} \label{thm:Meixner-Pollaczek functions}
	For $\phi(t) = \arccos(\frac{s(t)}{r(t)})$ let
	\[
	W_t(x) = W(x;\la,\eps,\phi(t)),
	\]
	and let $U_t : \ell^2(\Z) \to L^2(\R,W_t(x;\la,\eps)dx)$ be defined by
	\[
	[U_t e_n](x) = 
	\begin{pmatrix}
	u_n(x;\la,\eps,\phi(t)) \\ \overline{u_n}(x;\la,\eps,\phi(t))
	\end{pmatrix},
	\]
	then $U_t$ is unitary and $U_t \circ \pi_{\la,\eps}(L(t)) \circ U_t^{*} = M(2Cx)$, where $C =  \sqrt{r^2-s^2}$. 
\end{theorem}
Note that the spectrum of $\pi_{\la,\eps}(L(t))$ has multiplicity 2.

\begin{remark} Transferring a three-term recurrence on $\Z$ to a three term recurrence for 
$2\times 2$ matrix orthogonal polynomials, see \cite[\S VII.3]{Bere}, \cite[\S 3.2]{Koel},
does not lead to an example of the nonabelian Toda lattice \cite{BrusMRL}, \cite{Gekh}, \cite{IsmaKR}
\end{remark}

\section{The oscillator algebra $\mathfrak b(1)$}\label{sec:oscillatoralgebra}
$\mathfrak b(1)$ is the Lie $*$-algebra $\mathfrak g(a,b)$ with $(a,b)=(0,1)$ and $\epsilon=+$. Then $\mathfrak b(1)$ has a basis $E,F,H,N$ satisfying
\[
[E,F]=N, \quad [H,E]=2E, \quad [H,F]=-2F, \quad [N,E]=[N,F]=[N,H]=0.
\]
The $*$-structure is defined by $H^*=H$, $N^*=N$, $E^*=F$. 
The Lax pair $L,M$ is given by 
\[
L(t) = cH + r(t)(E+F) + s(t) N, \qquad M(t) = u(t) (E- F).
\]
The differential equations for $s$ and $r$ are in this case given by
	\[
	\begin{split}
	\dot s &= 2ru\\
	\dot r &= -2cu
	\end{split}
	\]
and the invariant is $r^2+2cs$.

\begin{lemma} \label{lem:sign r s for b(1)}
	Assume $\sgn(u(t))=\sgn(r(t)$ for all $t>0$, $s(0)>0$ and $r(0)>0$. Then $\sgn(s(t))>0$ and $\sgn(r(t))>0$ for all $t>0$. 
\end{lemma}
\begin{proof}
	The proof is similar to the proof of Lemma \ref{lem:sign r s for su(1,1)}, where in this case  $I(r,s)=I(r(0),s(0))$ describes a parabola ($c \neq 0$)  or a straight line ($c=0$).
\end{proof}
Throughout this section we assume the conditions of Lemma \ref{lem:sign r s for b(1)} are satisfied.\\

There is a family of irreducible $*$-representations $\pi_{k,h}$, $h>0$, $k\geq0$, on $\ell^2(\N)$  defined by
\[
\begin{split}
	\pi_{k,h}(N) e_n &= -h\, e_n\\
	\pi_{k,h}(H) e_n &= 2(k+n)\, e_n \\
	\pi_{k,h}(E) e_n &= \sqrt{h (n+1)} \, e_{n+1}\\
	\pi_{k,h}(F) e_n &= \sqrt{ hn}\, e_{n-1}.
\end{split}
\]
The action of the Lax operator on the basis of $\ell^2(\N)$ is given by
\[
 \pi_{k,h}(L(t)) e_n = r(t)\sqrt{h(n+1)}\, e_{n+1}+\left[2c(n+k) - h s(t) \right]e_n   + r(t)\sqrt{hn}\, e_{n-1}.
\]
For the diagonalization of $\pi_{k,h}(L)$ we distinguish between the cases $c \neq 0$ and $c = 0$. 

\subsection{Case 1: $c \neq 0$}
In this case we need the orthonormal Charlier polynomials \cite[Section 9.14]{KLS}, which are defined by
\[
C_n(x) = C_n(x;a) = \sqrt{ \frac{ a^n }{n!} } \rFs{2}{0}{-n,-x}{\mhyphen}{-\frac1a},
\]
where $a>0$ and $n,x \in \N$. 
The orthogonality relations are 
\[
\sum_{x=0}^\infty \frac{a^x e^{-a}}{x!} C_n(x) C_{n'}(x) = \de_{n,n'},
\]
and $\{C_n \mid n \in \N\}$ is an orthonormal basis for the corresponding $L^2$-space. The three-term recurrence relation reads
\[
-xC_n(x) = \sqrt{a(n+1)} \, C_{n+1}(x) - (n+a) C_n(x) + \sqrt{an} \, C_{n-1}(x).
\]

\begin{theorem} \label{thm:Charlier}
	For $t>0$ define
	\[
	W_t(x) = \left( \frac{h r^2(t) }{c} \right)^x e^{-\frac{hr^2(t)}{c^2}}
	\]
	and let $U_t:\ell^2(\N) \to L^2(\N, W_t)$ be defined by
	\[
	U_t e_n (x) = \left( -\sgn(r/c) \right)^n C_n\left(x ;\frac{hr^2(t)}{c^2}\right), \qquad x \in \N.
	\]
	Then $U_t$ is unitary and $U_t \circ L(t) \circ U_t^{*} = M(2c(x+k) + Ch)$, where $C=\frac1{2c}r^2+s$.  
\end{theorem}
\begin{proof}
The action of $L$ can be written in the following form,
\[
\begin{split}
	\pi_{k,h}\Big(\frac{1}{2c}L & + \frac{hr^2}{4c^2} + \frac{hs}{2c}-k \Big) e_n = \\
	&  \sgn(r/c) \sqrt{\frac{hr^2(n+1)}{4c^2}}\, e_{n+1} +\left(n + \frac{hr^2}{4c^2}\right) e_n + \sgn(r/c) \sqrt{\frac{hr^2 n}{4c^2} } e_{n-1}
\end{split}
\]
and recall  that $\frac{1}{2c}r^2+ s$ is constant. The result then follows from comparing with the three-term recurrence relation for the Charlier polynomials.
\end{proof}

\subsection{Case 2: $c=0$} In this case $\dot r=0$, so $r$ is a constant function. We use the orthonormal Hermite polynomials \cite[Section 9.15]{KLS}, which are given by
\[
H_n(x) = \frac{(\sqrt 2 \, x)^n }{\sqrt{n!}} \rFs{2}{0}{-\frac{n}{2}, - \frac{n-1}{2} }{\mhyphen}{-\frac{1}{x^2} }.
\]
They satisfy the orthogonality relations
\[
\frac{1}{\sqrt{\pi}} \int_\R H_n(x) H_{n'}(x) e^{-x^2}\, dx = \de_{n,n'},
\]
and $\{H_n \mid n \in \N\}$ is an orthonormal basis for $L^2(\R, e^{-x^2}dx/\sqrt{\pi})$. The three-term recurrence relation is given by
\[
\sqrt{2}\,x H_n(x) = \sqrt{n+1}\, H_{n+1}(x) + \sqrt{n} H_{n-1}(x).
\]

\begin{theorem} \label{thm:Hermite}
	For $t>0$ define
	\[
	W_t(x) = \frac{1}{r\sqrt{2h\pi}} e^{-\frac{(x-h s(t))^2}{2hr^2}},
	\]
	and let $U_t:\ell^2(\N) \to L^2(\R;w_t(x;h)\,dx)$ be defined by
	\[
	U_t e_n(x) = H_n\left(\frac{x-h s(t)}{r\sqrt{2h}}\right),
	\]
	then $U_t$ is unitary and $U_t \circ \pi_{k,h}(L(t)) \circ U_t^{*} = M(x)$.
\end{theorem}
\begin{proof}
	 We have
	\[
	\pi_{k,h}\left(\frac{1}{r \sqrt{h}}(L+sh)\right)e_n = \sqrt{n+1}\, e_{n+1} + \sqrt{n}\, e_{n-1},
	\]
	which corresponds to the three-term recurrence relation for the Hermite polynomials.
\end{proof}

\section{The Lie algebra $\mathfrak{e}(2)$}\label{sec:casee2}
We consider the Lie algebra $\mathfrak g(a,b)$ with $a=b=0$ and $\epsilon=+$. Similar as in the case of $\mathfrak{sl}(2,\C)$, we omit the basis element $N$ again. The remaining Lie algebra is $\mathfrak e(2)$ with basis $H, E, F$ satisfying
\[
[E,F]=0, \quad [H,E]=2E, \quad [H,F]=-2F,
\]
and the $*$-structure is determined by $E^*=F, \quad H^*=H$.

The Lax pair is given by
\[
L(t) = cH + r(t)(E+F), \qquad M(t) = u(t) (E-F),
\]
with $\dot r = -2cu$.

$\mathfrak e(2)$ has a family of irreducible $*$-representations $\pi_k$, $k>0$, on $\ell^2(\Z)$ given by
\[
\begin{split}
	\pi_k(H) e_n &= 2n\, e_n, \\
	\pi_k(E) e_n &= k e_{n+1}, \\
	\pi_k(F)e_n &= k e_{n-1}.
\end{split}
\]
This defines an unbounded representation. As a dense domain we use the set of finite linear combinations of the basis elements. 

Assume $c \neq 0$. The Lax operator $\pi_k(L(t))$ is a Jacobi operator on $\ell^2(\Z)$ given by
\[
\pi_k(L(t))e_n = kr(t) e_{n+1} + 2cn e_n + kr(t) e_{n-1}.
\]
For the diagonalization of $\pi_k(L)$ we use the Bessel functions $J_n$ \cite{Wat, AAR} given by
\[
J_n(z) = \frac{ z^n }{2^n \Ga(n+1) } \rFs{1}{0}{\mhyphen}{n+1}{-\frac{z^2}{4}},
\]
with $z \in \R$ and $n \in \Z$. They satisfy the Hansen-Lommel type orthogonality relations, which follow from \cite[(4.9.15), (4.9.16)]{AAR}
\[
\sum_{m \in \Z} J_{m-n}(z) J_{m-n'}(z) = \de_{n,n'}.
\]
and the set $\{J_{\cdot -n}(z) \mid  n \in \Z\}$ is an orthonormal basis for $\ell^2(\Z)$. A well-known recurrence relation for $J_n$ is
\[
J_{n-1}(z) + J_{n+1}(z) = \frac{2n}{z}J_n(z),
\]
which is equivalent to
\[
zJ_{m-n-1}(z)+2nJ_{m-n}(z) + z J_{m-n+1}(z) = 2m J_{m-n}(z).
\]

\begin{theorem}
	For $t>0$ define $U_t: \ell^2(\Z) \to \ell^2(\Z)$ by
	\[
	U_t e_n(m) = J_{m-n}\left(\frac{kr(t)}{c}\right),
	\]
	then $U_t$ is unitary and $U_t \circ \pi_k(L(t)) \circ U_t^{*} = M(2cm)$. 
\end{theorem}

Finally, let us consider the completely degenerate case $c=0$. In this case $r$ is also a constant function, so there are no differential equations to solve. We can still diagonalize the (degenerate) Lax operator, which is now independent of time. 
\begin{theorem}
	Define $U:\ell^2(\Z) \to L^2[0,2\pi]$ by
	\[
	[Ue_n](x) = \frac{e^{inx}}{\sqrt{2\pi}},
	\]
	then $U$ is unitary and $U \circ \pi_k(L)\circ U^* = M(2kr \cos(x))$. 
\end{theorem}

\section{Modification of orthogonality measures}\label{sec:modification}

In this section we briefly investigate the orthogonality measures from the previous sections in case the Lax operator $L(t)$ acts as a finite or semi-infinite Jacobi matrix. In these cases the functions $U_te_n$ are $t$-dependent orthogonal polynomials and we see that the weight function $W_t$ of the orthogonality measure for $U_t e_n$ is a modification of the weight function $W_0$ in the sense that
\[
W_t(x) = K_t W_0(x) m(t)^x,
\]
where $K_t$ is independent of $x$. The modification function $m(t)$ depends on the functions $s$ or $r$, which (implicitly) depend on the function $u$. We show how the choice of $u$ effects the modification function $m$.

\begin{theorem} \label{thm:modifications}
	There exists a constant $K$ such that 
	\[
	m(t) = \exp\left( K\int_0^t \frac{ u(\tau) }{r(\tau)}\, d\tau \right), \qquad t \geq 0.
	\]
\end{theorem}

\begin{remark}
In the Toda-lattice case, $u(t) = r(t)$, this gives back the well-known modification function $m(t) = e^{Kt}$, see e.g.~\cite[Theorem 2.8.1]{Isma}.
\end{remark}

Theorem \ref{thm:modifications} can be checked for each case by a straightforward calculation:
we express $m$ as a function of $s$ and $r$,
\[
m(t) = A_0 F(s(t),r(t)),
\] 
where $A_0$ is a normalizing constant such that $m(0)=1$. Then differentiating and using the differential equations for $r$ and $s$ we can express $\dot m/ m$ in terms of $u$ and $r$. 

\subsection{$\su(2)$}
From Theorem \ref{thm:Krawtchouk} we see that
\[
m(t) = A_0 \frac{p(t)}{1-p(t)} =  A_0 \frac{C+s(t)}{C-s(t)} 
\]
with $C = \sqrt{s^2+r^2}$. 
Differentiating to $t$ and using the relation $\dot s(t) = 2 u(t) r(t)$ then gives
\[
\frac{\dot m(t) }{m(t) } = \frac{ 4C u(t)r(t) }{C^2-s(t)^2} = 4C \frac{ u(t)}{r(t)}.
\]

\subsection{$\su(1,1)$} 
For $s^2-r^2>0$ Theorem \ref{thm:Meixner} shows that
\[
m(t) = A_0e^{-2\arccosh\left( \frac{s(t)}{r(t)} \right)}.
\]
Then from $\dot s(t) = -2u(t)r(t)$ and $\dot r(t) = -2u(t)s(t)$ it follows that
\[
\frac{\dot m(t) }{m(t) } = \frac{-2}{\sqrt{\frac{s(t)^2}{r(t)^2}-1}} \frac{ r(t) \dot s(t) - s(t) \dot r(t) }{r(t)^2} =  -4C \frac{u(t) }{r(t)},
\]
where $C = \sqrt{s^2-r^2}$.

For $s^2-r^2=0$ Theorem \ref{thm:Laguerre} shows that
\[
m(t)= A_0 e^{-\frac{1}{r(t)}}.
\]
Then using $\dot r(t) = 2u(t)r(t)$ it follows that
\[
\frac{\dot m(t) }{m(t)} = - \frac{u(t)}{r(t)} .
\]

For $s^2-r^2<0$ it follows from Theorem \ref{thm:Meixner-Pollaczek} that
\[
m(t) = A_0 e^{2\arccos\left( \frac{s(t) }{r(t)} \right)}.
\]
Then from $\dot s(t) = -2u(t)r(t)$ and $\dot r(t) = -2u(t)s(t)$ it follows that
\[
\frac{\dot m(t) }{m(t) } = \frac{2}{\sqrt{1-\frac{s(t)^2}{r(t)^2}}} \frac{ r(t) \dot s(t) - s(t) \dot r(t) }{r(t)^2} =  -4C \frac{u(t) }{r(t)},
\]
where $C = \sqrt{r^2-s^2}$.

\subsection{$\mathfrak b(1)$}
For $c \neq 0$ we see from Theorem \ref{thm:Charlier} that
\[
m(t) = A_0 r(t)^2.
\]
The relation $\dot r(t) = -2c u(t)$ then leads to
\[
\frac{ \dot m(t) }{m(t) } = -4c\frac{u(t)}{r(t)}.
\]
For $c=0$ Theorem \ref{thm:Hermite} shows that
\[
m(t) = A_0 e^{\frac{s(t)}r}.
\]
Note that $r=r(t)$ is constant in this case. Then $\dot s(t) = 2ru(t)$ leads to
\[
\frac{\dot m(t)}{m(t) }= 2u(t) = 2r \frac{ u(t) }{r}.
\]

\begin{remark}
	The result from Theorem \ref{thm:modifications} is also valid for the orthogonal functions from Theorems \ref{thm:Meixner functions} and \ref{thm:Meixner-Pollaczek functions}, i.e.~for $L(t)$ acting as a Jacobi operator on $\ell^2(\Z)$ in the principal unitary series for $\su(1,1)$ in cases $r^2-s^2 \neq 0$. However, there is no similar modification function in the other cases where $L(t)$ acts as a Jacobi operator on $\ell^2(\Z)$. Furthermore, the corresponding recurrence relations for the functions on $\Z$ can be rewritten to recurrence relations for $2\times 2$ matrix orthogonal polynomials, but in none of the cases the modification of the weight function is as in Theorem \ref{thm:modifications}.
\end{remark}

\section{The case of $\mathfrak{sl}(d+1,\C)$}\label{sec:sld+1}

We generalize the situation of the Lax pair for the finite-dimensional representation of
$\mathfrak{sl}(2,\C)$ to the higher rank case of $\mathfrak{sl}(d+1,\C)$. 
Let $E_{i,j}$ be the matrix entries forming a basis for the $\mathfrak{gl}(d+1,\C)$.
We label $i,j\in \{0,1,\cdots, d\}$. We put $H_i= E_{i-1,i-1}-E_{i,i}$, $i\in \{1,\cdots, d\}$,
for the elements spanning the Cartan subalgebra of $\mathfrak{sl}(d+1,\C)$.

\subsection{The Lax pair}

\begin{proposition}\label{prop:sld+1-Laxpair}
Let 
\begin{align*}
L(t) = \sum_{i=1}^d s_i(t) H_i + \sum_{i=1}^d r_i(t) \bigl( E_{i-1,i}+ E_{i,i-1}\bigr),
\qquad 
M(t) = \sum_{i=1}^d u_i(t) \bigl( E_{i-1,i} - E_{i,i-1} \bigr)
\end{align*}
and assume that the functions $u_i$ and $r_i$ are non-zero for all $i$ and 
\[
\frac{r_{i-1}(t)}{r_i(t)} =\frac{u_{i-1}(t)}{u_i(t)}, \qquad i\in \{2,\cdots, d\}
\]
then the Lax pair condition $\dot{L}(t)=[L(t),M(t)]$ is equivalent to 
\begin{align*}
\dot{s}_i(t) &= 2r_i(t) u_i(t), \qquad i\in \{1,\cdots, d\} \\
\dot{r}_i(t) &= u_i(t) \bigl( s_{i-1}(t) - 2s_i(t)  + s_{i+1}(t) \bigr), \qquad i\in \{2,\cdots, d-1\} \\ 
\dot{r}_1(t) &= u_1(t) \bigl(  s_{2}(t) - 2s_1(t) \bigr),  \\ 
\dot{r}_d(t) &= u_d(t) \bigl( s_{d-1}(t) -2s_d(t)  \bigr).
\end{align*}
\end{proposition}

Note that we can write it uniformly 
\[
\dot{r}_i(t) = u_i(t) \bigl( s_{i-1}(t) - 2s_i(t) + s_{i+1}(t)\bigr), \qquad i\in \{1,\cdots, d\} \\ 
\]
assuming the convention that $s_0(t)=s_{d+1}(t)=0$, which we adapt for the remainder of this section. 
The Toda case follows by taking $u_i=r_i$ for all $i$, see \cite{BabeBT}, \cite{Mose}.

\begin{proof}
The proof essentially follows as in \cite[\S 4.6]{BabeBT}, but since the situation is slightly more general
we present the proof, see also \cite[\S 5]{Mose}. A calculation in $\mathfrak{sl}(d+1,\C)$ gives
\begin{multline*}
[M(t),L(t)]  = \sum_{i=1}^d 2r_i(t)u_i(t) H_i + 
\sum_{i=1}^d u_i(t) \bigl( s_{i-1}(t) - 2s_i(t) + s_{i+1}(t) \bigr)(E_{i-1,i}+E_{i,i-1}) \\ 
+ \sum_{i=1}^{d-1} \bigl(r_{i+1}(t) u_i(t) - r_{i}(t) u_{i+1}(t) \bigr) (E_{i-2,i}+E_{i,i-2})
\end{multline*}
and the last term needs to vanish, since this term does not occur in $L(t)$ and in its derivative
$\dot{L}(t)$. Now the stated coupled differential equations correspond to $\dot{L}=[M,L]$.
\end{proof}

\begin{remark}\label{rmk:Ltinsld+1fromKrawtchouk}
Taking the representation of the Lax pair for the $\mathfrak{su}(2)$ case in the $d+1$-dimensional 
representation as in Section \ref{sec:casee2}, we get, with $d=2j$, as an example 
\[
s_i(t) = s(t)i(i-1-d), \quad 
r_i(t) = r(t)\sqrt{i(d+1-i)}, \quad
u_i(t) = u(t)\sqrt{i(d+1-i)}.
\]
Then the coupled differential equations of Proposition \ref{prop:sld+1-Laxpair} are equivalent to 
\eqref{eq:relations r s u su(2)}. 
\end{remark}

Let $\{e_n\}_{n=0}^d$ be the standard orthonormal basis for $\C^{d+1}$, the natural representation 
of $\mathfrak{sl}(d+1,\C)$. 
Then $L(t)$ is a $t$-dependent tridiagonal matrix. Moreover, we assume that $r_i$ and $s_i$ are real-valued functions for all $i$, so that $L(t)$ is self-adjoint in the natural representation.
 
\begin{lemma}\label{lem:polsLsld+1}
Assume that the conditions of Proposition \ref{prop:sld+1-Laxpair} hold. 
Let the polynomials $p_n(\cdot;t)$ of degree $n \in \{0,1,\cdots, d\}$ in $\la$ be generated
by the initial value $p_0(\la;t)=1$ and the recursion
\[
\la p_n(\la;t) = \begin{cases} r_1(t) p_1(\la;t) + s_1(t) p_0(\la;t), & n=0 \\
r_{n+1}(t) p_{n+1}(\la;t) + (s_{n+1}(t)-s_n(t)) p_n(\la;t) + r_n(t) p_{n-1}(t), & 1\leq n < d.              
\end{cases}
\]
Let the set $\{ \la_0, \cdots, \la_d\}$ be the zeroes of 
\[
\la p_d(\la;t) = -s_d(t) p_d(\la;t) + r_d(t) p_{d-1}(t). 
\]
In the natural representation $L(t)$ has simple spectrum $\si(L(t))= \{ \la_0, \cdots, \la_d\}$ which 
is independent of $t$, and $\sum_{r=0}^d \la_r=0$ and 
\[
L(t) \sum_{n=0}^d p_n(\la_r;t)e_n = \la_r \, \sum_{n=0}^d p_n(\la_r;t)e_n, \quad r\in \{0,1\cdots, d\}.
\]
\end{lemma}

Note that with the choice of Remark \ref{rmk:Ltinsld+1fromKrawtchouk}, the polynomials in 
Lemma \ref{lem:polsLsld+1} are Krawtchouk polynomials, see Theorem \ref{thm:Krawtchouk}. 
Explicitly,
\begin{equation}\label{eq:rmk:Ltinsld+1fromKrawtchouk}
p_n(C(d-2r);t) =  \left( \frac{p(t)}{1-p(t)}\right)^{\frac12 n} \binom{d}{n}^{1/2}
\rFs{2}{1}{-n, -r}{-d}{\frac{1}{p(t)}} = K_n(r;p(t),d),
\end{equation}
where $C=\sqrt{r^2(t)+s^2(t)}$ is invariant, see 
Theorem \ref{thm:Krawtchouk} and its proof. 

\begin{proof} In the natural representation we have
\[
L(t) e_n = 
\begin{cases} 
r_{1}(t) e_1 + s_1(t) e_0 & n=0 \\  
r_{n+1}(t)e_{n+1} + (s_{n+1}(t)-s_n(t))e_n +r_{n-1}(t) e_{n-1}  & 1\leq n <d \\
-s_d(t) e_d + r_d(t) e_{d-1} & n= d
\end{cases}
\]
as a Jacobi operator. So the spectrum of $L(t)$ is simple, and  the 
spectrum is time independent, since $(L(t),M(t))$ is a Lax pair. 
We can generate the corresponding eigenvectors as 
$\sum_{n=0}^d p_n(\la;t) e_n$, where the recursion follows from the expression 
of Lemma \ref{lem:polsLsld+1}. The eigenvalues are then determined by the final equation, and
since $\mathrm{Tr}(L(t))=0$ we have $\sum_{i=0}^d \la_i=0$. 
\end{proof}

Let $P(t) = \bigl(p_i(\la_j;t)\bigr)_{i,j=0}^d$ be the corresponding matrix of eigenvectors, so that 
\[
L(t) P(t) = P(t) \La, \qquad \La = \mathrm{diag}(\la_0, \la_1,\cdots, \la_d). 
\]
Since $L(t)$ is self-adjoint in the natural representation, we find
\begin{equation}\label{eq:orthorelpolLtsld+1}
\sum_{n=0}^d p_n(\la_r;t) \overline{p_n(\la_s;t)} = \frac{\de_{r,s}}{w_r(t)}, \qquad w_r(t)>0,
\end{equation}
and the matrix $Q(t) = \bigl(p_i(\la_j;t)\sqrt{w_j(t)} \bigr)_{i,j=0}^d$ is unitary. 
As $r_i$ and $s_i$ are real-valued, we have $\overline{p_n(\la_s;t)} = p_n(\la_s;t)$, so that 
$Q(t)$ is a real matrix, hence orthogonal. So the dual orthogonality relations 
to \eqref{eq:orthorelpolLtsld+1} hold as well. We will assume moreover that $r_i$ are positive
functions. The dual orthogonality relations to \eqref{eq:orthorelpolLtsld+1} hold;
\begin{equation}\label{eq:dualorthorelpolLtsld+1}
\sum_{r=0}^d p_n(\la_r;t) p_m(\la_r;t) w_r(t) = \de_{n,m}.
\end{equation}
Note that the $w_r(t)$ are essentially time-dependent Christoffel numbers \cite[\S 3.4]{Szeg}. 
By \cite[\S 2]{Mose}, see also \cite[Thm.~2]{DeifNT},  the eigenvalues and the $w_r(t)$'s determine the operator $L(t)$, and in case
of the Toda lattice, i.e. $u_i(t) = r_i(t)$, the time evolution corresponds to linear first order 
differential equations for the Christoffel numbers \cite[\S 3]{Mose}.

Since the spectrum is time-independent, the invariants for the 
system of Proposition  \ref{prop:sld+1-Laxpair} are given by the coefficients 
of the characteristic polynomial of $L(t)$ in the natural representation. Since
the characteristic polynomial is obtained by switching to the three-term recurrence
for the corresponding monic polynomials, see \cite[\S 2.2]{Isma}, \cite[\S 2]{Mose}, this gives the same
computation. 
For a Lax pair, $\mathrm{Tr}(L(t)^k)$ are invariants, and in this case the invariant for 
$k=1$ is trivial since $L(t)$ is traceless. In this way we have $d$ invariants,
$\mathrm{Tr}(L(t)^k)$, $k\in \{2,\cdots, d+1\}$. 

\begin{lemma}\label{lem:Lsld+1-invariants} With the convention that $r_n$ and $s_n$ are 
zero for $n\notin\{1,\cdots,d\}$ we have the invariants
\begin{align*}
\mathrm{Tr}(L(t)^2) &= \sum_{n=0}^d (s_{n+1}(t)-s_n(t))^2 + \sum_{n=1}^d r_n(t)^2 \\
\mathrm{Tr}(L(t)^3) &= \sum_{n=0}^d (s_{n+1}(t)-s_n(t))^3 
+ 3\sum_{n=0}^d (s_{n+1}(t)-s_n(t)) r_n^2(t) \\ 
&\qquad \qquad + 
3\sum_{n=0}^d (s_{n}(t)-s_{n-1}(t)) r_n^2(t)
\end{align*}
\end{lemma}

\begin{proof} Write $L(t) = DS + D_0 + S^\ast D$ with 
$D=\mathrm{diag}(r_0(t), r_1(t),\cdots, r_d(t))$, $S\colon e_n\mapsto e_{n+1}$ the shift 
operator and $S^\ast \colon e_n\mapsto e_{n-1}$ its 
adjoint (with the convention $e_{-1}=e_{d+1}=0$ and $r_0(t)=0$). And
$D_0$ is the diagonal part of $L(t)$. 
Then 
\[
\mathrm{Tr}(L(t)^k) = \mathrm{Tr}((DS + D_0 + S^\ast D)^k)
\]
and we need to collect the terms that have the same number of $S$ and $S^\ast$ in the 
expansion. 
The trace property then allows to collect terms, and we get
\begin{align*}
\mathrm{Tr}(L(t)^2) &= \mathrm{Tr}(D_0^2) + 2\mathrm{Tr}(D^2), \\
\mathrm{Tr}(L(t)^3) &= \mathrm{Tr}(D_0^3) + 3\mathrm{Tr}(D_0D^2) +  3\mathrm{Tr}(SD_0S^\ast D^2)
\end{align*}
and this gives the result, since $(SD_0S^\ast)_{n,n}= (D_0)_{n-1,n-1}$. 
\end{proof}

We do not use Lemma \ref{lem:Lsld+1-invariants}, and we have included to indicate the analog of
Corollary \ref{cor:constant function}.

We can continue this and find e.g.
\begin{align*}
\mathrm{Tr}(L(t)^4) &= \mathrm{Tr}(D_0^4) + 2\mathrm{Tr}(D^4) + 4\mathrm{Tr}(D_0^2D^2) 
+ 4\mathrm{Tr}(SD_0S^\ast D_0D^2) \\ 
& \qquad + 4\mathrm{Tr}(SD_0^2S^\ast D^2)
+ 4\mathrm{Tr}(SD^2S^\ast D^2)
\end{align*}

\subsection{Action of $L(t)$ in representations}

We relate the eigenvectors of $L(t)$ in some explicit representations of 
$\mathfrak{sl}(d+1)$ to multivariable Krawtchouk polynomials, and we follow 
Iliev's paper \cite{Ilie}.

Let $N\in \N$, and let $\C_N[x]=\C_N[x_0,\cdots, x_d]$ be the space 
of homogeneous polynomials of degree $N$ in $d+1$-variables, then $\C_N[x]$ is an
irreducible representation of $\mathfrak{sl}(d+1)$ and 
$\mathfrak{gl}(d+1)$ given by $E_{i,j} \mapsto x_i \frac{\partial}{\partial x_j}$. 
$\C_N[x]$ is a highest weight representation corresponding to $N\om_1$, $\om_1$ 
being the first fundamental weight for type $A_d$.
Then $x^\rho = x_0^{\rho_0}\cdots x_d^{\rho_d}$, $|\rho|=\sum_{i=0}^d\rho_i=N$, is  
an eigenvector of $H_i$; $H_i\cdot x^\rho = (\rho_{i-1}-\rho_i)x^\rho$, and 
so we have a basis of joint eigenvectors of the Cartan subalgebra spanned by
$H_1,\cdots, H_d$ and the joint eigenspace, i.e. the weight space, is $1$-dimensional.
It is a unitary representation for the inner product
\[
\langle x^\rho, x^\si \rangle = \de_{\rho,\si} \binom{N}{\rho}^{-1}= 
\de_{\rho,\si} \frac{\rho_0! \cdots \rho_d!}{N!}
\]
and it gives a unitary representation of $SU(d+1)$ as well. 

Then the eigenfunctions of $L(t)$ in $\C_N[x]$ are 
$\tilde{x}^\rho$, where 
\[
(\tilde{x}_0, \cdots, \tilde{x}_d) = (x_0, \cdots, x_d) Q(t)
\]
since $Q(t)$ changes from eigenvectors for the Cartan subalgebra to 
eigenvectors for the operator $L(t)$, cf. \cite[\S 3]{Ilie}. It corresponds to 
the action of $SU(d+1)$ (and of $U(d+1)$) on $\C_N[x]$. 
Since $Q(t)$ is unitary, we have 
\begin{equation}\label{eq:sld+1orthrelevLt}
\langle \tilde{x}^\rho, \tilde{x}^\si\rangle = 
\langle x^\rho,  x^\si\rangle = \de_{\rho,\si} \binom{N}{\rho}^{-1}.
\end{equation}

We recall the generating function for the multivariable Krawtchouk polynomials
as introduced by Griffiths \cite{Grif}, see \cite[\S 1]{Ilie}: 
\begin{equation}\label{eq:sld+1genfunmvKrawtchouk}
\prod_{i=0}^d \Bigl(z_0 + \sum_{j=1}^d u_{i,j} z_j\Bigr)^{\rho_i} 
= \sum_{|\si|=N} \binom{N}{\si} P(\si',\rho') z_0^{\si_0}\cdots z_d^{\si_d}
\end{equation}
where $\rho'= (\rho_1,\cdots, \rho_d) \in \N^d$, and similarly for $\si'$. 
We consider $P(\rho',\si')$ as polynomials in $\si'\in \N^d$ of degree $\rho'$ 
depending on $U=(u_{i,j})_{i,j=1}^d$, see \cite[\S 1]{Ilie}. 

\begin{lemma}\label{lem:sld+1eigvectL} 
The eigenvectors of $L(t)$ in $\C_N[x]$ are
\[
\tilde{x}^\rho = \prod_{i=0}^d \bigl( w_i(t)\bigr)^{\frac12 \rho_i} 
\sum_{|\si|=N} \binom{N}{\si} P(\si',\rho') x^\si
\]
for $u_{i,j} = \frac{Q(t)_{j,i}}{Q(t)_{0,i}}= p_j(\la_i;t)$, $1 \leq i,j\leq d$ 
in \eqref{eq:sld+1genfunmvKrawtchouk}, and $L(t) \tilde{x}^\rho =
(\sum_{i=0}^d \la_i \rho_i )\tilde{x}^\rho$. The eigenvalue follows from 
the conjugation with the diagonal element $\La$. 
\end{lemma}

From now on we assume this value for $u_{i,j}$, $1\leq i,j\leq d$. Explicit expressions
for $P(\si',\rho')$ in terms of Gelfand hypergeometric series are due to Mizukawa and Tanaka \cite{MizuT},
see \cite[(1.3)]{Ilie}. See also Iliev \cite{Ilie} for an overview of special and related cases 
of the multivariable cases. 

\begin{proof} Observe that 
\[
\tilde{x}_i = \sum_{j=0}^d x_j Q(t)_{j,i} = Q(t)_{0,i} \Bigl( x_0 
+ \sum_{j=1}^d \frac{Q(t)_{j,i}}{Q(t)_{0,i}} x_j\Bigr)
\]
and $Q(t)_{0,i}= \sqrt{w_i(t)}$ is non-zero. Now expand $\tilde{x}^\rho$ 
using \eqref{eq:sld+1genfunmvKrawtchouk} and $Q(t)_{i,j} = p_i(\la_j;t) \sqrt{w_j(t)}$
gives the result. 
\end{proof}

By the orthogonality \eqref{eq:sld+1orthrelevLt} 
of the eigenvectors of $L(t)$ we find
\begin{align*}
\sum_{|\si|=N} &\binom{N}{\si} P(\si',\rho')P(\si',\eta') = 
\frac{\de_{\rho,\eta}}{\binom{N}{\rho} \prod_{i=0}^d w_i(t)^{\rho_i}}, \\
\sum_{|\rho|=N} &\binom{N}{\rho} \Bigl( \prod_{i=0}^d w_i(t)^{\rho_i}\Bigr) P(\si',\rho')P(\tau',\rho') = 
\frac{\de_{\si,\tau}}{\binom{N}{\si}}, 
\end{align*}
where we use that all entries of $Q(t)$ are real. The second orthogonality follows 
by duality, and the orthogonality corresponds to \cite[Cor.~5.3]{Ilie}.

In case $N=1$ we find $P(f_i',f_j')= p_i(\la_j;t)$, where $f_i\in \N^{d+1}$ is given
by $(0,\cdots, 0, 1,0\cdots, 0)$ with the $1$ on the $i$-th spot. 

\begin{lemma}\label{lem:recursioneqPtaurho}
For all $\rho, \tau \in \N^{d+1}$ with $|\rho|=|\tau|$ we have
for the $P$ from Lemma \ref{lem:sld+1eigvectL} the recurrence
\begin{gather*}
\Bigl( \sum_{i=0}^d \la_i\rho_i\Bigr) P(\tau',\rho') 
= \Bigl( \sum_{i=0}^d s_i(t) (\tau_{i-1}-\tau_i)\Bigr) P(\tau',\rho')\\
+ \sum_{i=0}^d r_i(t) \bigl( \tau_{i-1} P((\tau-f_{i-1}+f_i)',\rho')
+ \tau_{i} P((\tau+f_{i-1}-f_i)',\rho')\bigr) 
\end{gather*}
\end{lemma}

Note that Lemma \ref{lem:recursioneqPtaurho} does not follow from 
\cite[Thm.~6.1]{Ilie}.

\begin{proof}
Apply Lemma \ref{lem:sld+1eigvectL} to expand $\tilde{x}^\rho$ in 
$L(t)\tilde{x}^\rho = (\sum_{i=0}^d \la_i\rho_i) \tilde{x}^\rho$, and use the explicit expression 
of $L(t)$ and the corresponding action. Compare the coefficient of $x^\tau$ 
on both sides to obtain the result.
\end{proof}

\begin{remark}\label{rmk:Ltinsld+1fromKrawtchouk2} In the 
context of Remark \ref{rmk:Ltinsld+1fromKrawtchouk} and \eqref{eq:rmk:Ltinsld+1fromKrawtchouk} 
we have that the $u_{i,j}$ are Krawtchouk polynomials. Then the left hand side in 
\eqref{eq:sld+1genfunmvKrawtchouk} is related to the generating function for the 
Krawtchouk polynomials, see \cite[(9.11.11)]{KLS}, i.e. the case $d=1$ of 
\eqref{eq:sld+1genfunmvKrawtchouk}. Putting $z_j = (\frac{p}{1-p})^{-\frac12 j}
\binom{d}{j}^{\frac12}w^j$, we see that in this
situation $\sum_{j=0}^d u_{i,j} z_j$ corresponds to $(1+w)^{d-i} (1- \frac{1-p(t)}{p(t)}w)^i$. 
Using this in the generating function, the left hand side of \eqref{eq:sld+1genfunmvKrawtchouk}
gives a generating function for Krawtchouk polynomials. Comparing the powers of $w^k$ on both sides gives
\begin{multline*}
\left( \frac{p}{1-p} \right)^{\frac12 k} 
 \binom{dN}{k} \rFs{2}{1}{-\sum_{i=0}^di\rho_i, -k}{-dN}{\frac{1}{p}} = \\
\sum_{|\si|=N, \sum_{j=0}^d j\si_j=k} 
\left(\prod_{j=0}^d \binom{d}{j}^{\frac12 \si_j}\right)
\binom{N}{\si} P(\si',\rho').
\end{multline*}
The left hand side is, up to a normalization, the overlap coefficient of $L(t)$ in the 
$\mathfrak{sl}(2,\C)$ case for the representation of dimension $Nd+1$, see \S \ref{sec:su2}. 
Indeed, the representation $\mathfrak{sl}(2,\C)$ to $\mathfrak{sl}(2,\C)$ 
to $\textrm{End}(\C_N[x])$ yields a reducible representation of $\mathfrak{sl}(2,\C)$, and the vector 
$x^{(0,\cdots, 0,N)}$ is a highest weight vector of $\mathfrak{sl}(2,\C)$ for the highest weight $dN$.
Restricting to this space then gives the above connection. 
\end{remark}

\subsection{$t$-Dependence of multivariable Krawtchouk polynomials}

Let $L(t) v(t) = \la v(t)$, then taking the $t$-derivatives gives
$\dot{L}(t)v(t) + L(t)\dot{v}(t) =\la \dot{v}(t)$, since $\la$ is 
independent of $t$, and using the 
Lax pair $\dot{L}=[M,L]$  gives 
\[
(\la - L(t)) (M(t) v(t) -\dot{v}(t))=0.
\]
Since $L(t)$ has simple spectrum, we conclude that 
\[
M(t) v(t) = \dot{v}(t) + c(t,\la) v(t) 
\]
for some constant $c$ depending on the eigenvalue $\la$ and $t$.
Note that this differs from \cite[Lemma~2]{Pehe}.

For the case $N=1$ we get 
\[
M(t)v_{\la_r}(t) = \sum_{n=0}^d \bigl( p_{n-1}(\la_r;t) u_n(t) - p_{n+1}(\la_r;t) u_{n+1}(t)\bigr) x_n
\]
with the convention that $u_0(t)=u_{d+1}(t)=0$, $p_{-1}(\la_r;t)=0$. 
So 
\[
(M(t)-c(t,\la_r))v_{\la_r}(t) =  \dot{v}_{\la_r}(t) = \sum_{n=0}^d \dot{p}_n(\la_r;t) \, x_n
\]
and comparing the coefficient of $x_0$, we find 
$c(t,\la_r) = - p_1(\la_r;t)u_1(t)$.
So we have obtained the following proposition. 

\begin{proposition} The polynomials satisfy
\begin{equation*}
\begin{split}
\dot{p}_n(\la_r;t) &= u_n(t) p_{n-1}(\la_r;t) - u_{n+1}(t) p_{n+1}(\la_r;t) + u_1(t) p_1(\la_r;t) p_n(\la_r;t), 
\qquad 1\leq n<d \\
\dot{p}_d(\la_r;t) &= u_d(t) p_{d-1}(\la_r;t) + u_1(t) p_1(\la_r;t) p_d(\la_r;t)
\end{split}
\end{equation*}
for all eigenvalues $\la_r$ of $L(t)$, $r\in \{0,\cdots, d\}$. 
\end{proposition}

Note that for $0 \leq n<d$ we have 
\begin{equation}\label{eq:derivpnlat-N=1}
\dot{p}_n(\la;t) = u_n(t) p_{n-1}(\la;t) - u_{n+1}(t) p_{n+1}(\la;t) + u_1(t) p_1(\la;t) p_n(\la;t) 
\end{equation}
as polynomial identity. Indeed, for $n=0$ this is trivially satisfied, and for $1\leq n<d$,
this is a polynomial identity of degree $n$ due to the condition in 
Proposition \ref{prop:sld+1-Laxpair}, which holds for all $\la_r$ and hence is a polynomial
identity. Note that the right hand side is a polynomial of degree $n$, and not of degree $n+1$
since the coefficient of $\la^{n+1}$ is zero because of the relation on $u_i$ and $r_i$ 
in Proposition \ref{prop:sld+1-Laxpair}. 

Writing out the identity for the Krawtchouk polynomials we obtain after simplifying
\begin{gather*}
n \rFs{2}{1}{-n,-r}{-d}{\frac{1}{p(t)}} 
+ \frac{2nr(1-p(t))}{dp(t)}\rFs{2}{1}{1-n,1-r}{1-d}{\frac{1}{p(t)}}  = \\
 n(1-p(t)) \rFs{2}{1}{1-n,-r}{-d}{\frac{1}{p(t)}}
- p(t) (d-n) \rFs{2}{1}{-1-n,-r}{-d}{\frac{1}{p(t)}}\\
+ (dp(t)-r) \rFs{2}{1}{-n,-r}{-d}{\frac{1}{p(t)}},
\end{gather*}
where the left hand side is related to the derivative. Note that the derivative of 
$p$ cancels with factors $u$, see Theorem \ref{thm:Krawtchouk} and its proof and 
\S \ref{sec:modification}.

In order to obtain a similar expression for the multivariable $t$-dependent Krawtchouk polynomials
we need to assume that the spectrum of $L(t)$ is simple, i.e. we assume that 
for $\rho,\tilde{\rho} \in \N^{d+1}$ with $|\rho|=|\tilde{\rho}|$ we have 
that $\sum_{i=0}^d \la_i(\rho_i-\tilde{\rho}_i)=0$ implies $\rho=\tilde{\rho}$. 
Assuming this we calculate, using Proposition \ref{prop:sld+1-Laxpair}, 
\begin{equation*}
M(t) \tilde{x}^\rho = W_\rho(t) 
\sum_{|\si|=N} \binom{N}{\si} P(\si',\rho') 
\sum_{r=1}^d u_r(t) (\si_r x^{\si+f_{r-1}-f_r} - \si_{r-1} x^{\si-f_{r-1}+f_r})
\end{equation*}
using the notation $W_\rho(t) = \prod_{i=0}^d w_i(t)^{\frac12 \rho_i}$ and $f_i=(0,\cdots,0,1, 0,\cdots, 0)\in \N^{d+1}$, with the $1$ at the $i$-th spot. Now the $t$-derivative of $\tilde{x}^\rho$ is 
\begin{equation*}
\dot{W}_\rho(t) \sum_{|\si|=N} \binom{N}{\si} P(\si',\rho') x^\si + 
W_\rho(t) \sum_{|\si|=N} \binom{N}{\si} \dot{P}(\si',\rho') x^\si 
\end{equation*}
and it leaves to determine the constant in 
$M(t) \tilde{x}^\rho - C\tilde{x}^\rho = \frac{\partial}{\partial t}\tilde{x}^\rho$. We determine
$C$ by looking at the coefficient of $x_0^N$ using $P(0, \rho')= P((N,0,\cdots,0)',\rho')=1$. This gives
$C= N u_1(t)W_\rho(t)^{-1} - \frac{\partial}{\partial t} \ln W_\rho(t)$.
Comparing the coefficients of $x^\tau$ on both sides gives the following result. 

\begin{theorem} Assume that $L(t)$ acting in $\C_N[x]$ has simple spectrum.
The $t$-derivative of the multivariable Krawtchouk polynomials satisfies 
\begin{gather*}
\dot{W}_\rho(t) P(\tau',\rho') + W_\rho(t) \dot{P}(\tau',\rho')= 
\bigl( \dot{W}_\rho(t) -N u_1(t)\bigr) P(\tau',\rho') + \\
W_\rho(t) \sum_{r=1}^d u_r(t) 
\bigl( \tau_{r-1} P((\tau-f_{r-1}+f_r)',\rho') - 
\tau_{r} P((\tau+f_{r-1}-f_r)',\rho')\bigr)
\end{gather*}
for all $\rho,\tau\in \N^{d+1}$, $|\tau|=|\rho|=N$. 
\end{theorem}


\begin{thebibliography}{99}
	
\bibitem{AAR} G.E. Andrews, R. Askey, R. Roy, \textit{Special Functions}, Encycl. Math. Appl. 71, Cambridge Univ. Press, 1999.
	
	\bibitem{BabeBT}
	O.~Babelon, D.~Bernard, M.~Talon, 
	\emph{Introduction to classical integrable systems}, 
	Cambridge Univ. Press, 2003. 

\bibitem{Bere}
Ju.M.~Berezanski\u\i , 
\emph{Expansions in eigenfunctions of selfadjoint operators},
Translations of Math. Monographs \textbf{17}, AMS, 1968.  

\bibitem{BrusMRL}
M.~Bruschi, S.V.~Manakov, O.~Ragnisco, D.~Levi, 
\emph{The nonabelian Toda lattice-discrete analogue of the matrix Schr\"odinger spectral problem},
J. Math. Phys. \textbf{21} (1980), 2749--2753.

\bibitem{CramvdVV}
N.~Cramp\'e, W.~van de Vijver, L.~Vinet, 
\emph{Racah problems for the oscillator algebra, the Lie algebra $\mathfrak{sl}_n$, and multivariate Krawtchouk polynomials}, Ann. Henri Poincar\'e \textbf{21} (2020), 3939--3971.

\bibitem{DeifNT}
P.~Deift, T.~Nanda, C.~Tomei, 
\emph{Ordinary differential equations and the symmetric eigenvalue problem},
SIAM J. Numer. Anal. \textbf{20} (1983), 1--22. 

\bibitem{Gekh}
M. Gekhtman,
\emph{Hamiltonian structure of non-abelian Toda lattice},
Lett. Math. Phys. \textbf{46} (1998), 189--205.

\bibitem{GeneVZ}
V.X.~Genest, L.~Vinet, Luc; A.~Zhedanov, 
\emph{The multivariate Krawtchouk polynomials as matrix elements of the rotation group representations on oscillator states}, J. Phys. A \textbf{46} (2013), 505203, 24 pp.

	\bibitem{Grif}
	R.C.~Griffiths,
	\emph{Orthogonal polynomials on the multinomial distribution}, 
	Austral. J. Statist. \textbf{13} (1971), 27--35. 
	
	\bibitem{Groe2003} W.~Groenevelt, \textit{Laguerre functions and representations of $su(1,1)$}, Indag.~Math. (N.S.) \textbf{14} (2003), 329--352
	
	\bibitem{GroeK2002}
	W.~Groenevelt, E.~Koelink, 
	\emph{Meixner functions and polynomials related to Lie algebra representations}, 
	J. Phys. A \textbf{35} (2002), 65--85. 
	
	\bibitem{Ilie}
	P.~Iliev, \emph{A Lie-theoretic interpretation of multivariate hypergeometric polynomials}, 
	Compos. Math. \textbf{148} (2012), 991--1002. 
	
	\bibitem{Isma}
	M.E.H.~Ismail, 
	\emph{Classical and quantum orthogonal polynomials in one variable},
	Encycl. Math. Appl. 98, Cambridge Univ. Press, 2005.

	\bibitem{IsmaKR}	
M.E.H.~Ismail, E.~Koelink, P.~Rom\'an, 
\emph{Matrix valued Hermite polynomials, Burchnall formulas and non-abelian Toda lattice}, 
Adv. in Appl. Math. \textbf{110} (2019), 235--269. 

\bibitem{Kame}
Y.~Kametaka, \emph{On the Euler–Poisson–Darboux equation and the Toda equation. I, II}, Proc. Japan Acad. Ser. A
Math. Sci. 60 (1984), 145--148, 181--184.

	\bibitem{KLS}  R.~Koekoek, P.A.~Lesky, R.~Swarttouw, \textit{Hypergeometric orthogonal polynomials and their q-analogues}, Springer Monographs in Math., Springer, 2010.

\bibitem{Koel2004} E.~Koelink, \textit{Spectral theory and special functions}, 
pp.45--84 in ``Laredo Lectures on Orthogonal Polynomials and Special Functions'' (eds. R.~\'Alvarez-Nodarse, F.~Marcell\'an, W.~Van Assche), Adv. Theory Spec. Funct. Orthogonal Polynomials, Nova Sci. Publ., 2004.

\bibitem{Koel}
E.~Koelink,
\emph{Applications of spectral theory to special functions}, pp.~131--212 in ``Lectures on Orthogonal Polynomials and Special Functions'' (eds. H.S. Cohl, M.E.H. Ismail), Lecture Notes of the London Math. Soc. \textbf{464}, Cambridge U. Press, 2021.  

	\bibitem{KVdJ98} H.T.~Koelink, J.~Van Der Jeugt, \textit{Convolutions for orthogonal polynomials from Lie and quantum algebra representations}, SIAM J.~Math.~Anal.~\textbf{29} (1998), 794--822.
	
	\bibitem{Mi} W.~Miller Jr., \textit{Lie theory and special functions}, Math. in Science and Engineering \textbf{43}, Academic Press, 1968.
		
\bibitem{MizuT}
H.~Mizukawa, H.~Tanaka, 
\emph{$(n+1,m+1)$-hypergeometric functions associated to character algebras},
Proc. Amer. Math. Soc. \textbf{132} (2004), 2613--2618. 
	
	
\bibitem{Mose}
J.~Moser, \emph{Finitely many mass points on the line under the influence of an exponential potential -- an integrable system}, pp. 467--497 in  ``Dynamical systems, theory and applications'' (ed. J.~Moser), 
Lecture Notes in Phys. \textbf{38}, Springer, 1975.

\bibitem{Okam}
K.~Okamoto, \emph{Sur les \'echelles associ\'ees aux fonctions sp\'eciales et l'\'equation de Toda}, J. Fac. Sci. Univ. Tokyo Sect. IA Math. \textbf{34} (1987), 709--740. 

\bibitem{Pehe}
F.~Peherstorfer, 
\emph{On Toda lattices and orthogonal polynomials}
J. Comput. Appl. Math. \textbf{133} (2001), 519--534. 

\bibitem{Sl} L.J.~Slater, \textit{Confluent hypergeometric functions}, Cambridge Univ. Press, 1960.


\bibitem{Szeg}
G.~Szeg\H{o}, \emph{Orthogonal polynomials}, 4th ed., Colloquium Publ. \textbf{23}, AMS, 1975.

\bibitem{Tesc}
G.~Teschl, 
\emph{Almost everything you always wanted to know about the Toda equation}, 
Jahresber. Deutsch. Math.-Verein. \textbf{103} (2001), 149--162. 

\bibitem{Wat} G.N.~Watson, \textit{A Treatise on the Theory of Bessel Functions}, Cambridge Univ. Press, 1944.
	
\bibitem{Zhed}
A.S.~Zhedanov, 
\emph{Toda lattice: solutions with dynamical symmetry and classical orthogonal polynomials},
Theoret. and Math. Phys. \textbf{82} (1990), 6--11. 

\end{thebibliography}
\end{document}